\documentclass{amsart}     
\usepackage{amssymb} 
\usepackage{amsmath,amscd,eucal,latexsym} 
\usepackage{graphicx}

\newtheorem{theorem}{Theorem}[section]
\newtheorem{lemma}[theorem]{Lemma}
\newtheorem{cor}[theorem]{Corollary}
\newtheorem{prop}[theorem]{Proposition}

\theoremstyle{definition} 
\newtheorem*{rem}{Remark}

\newtheorem*{conv}{Convention}
\newtheorem*{nota}{Notation}
\newtheorem{example}[theorem]{Example}
\newtheorem{defn}[theorem]{Definition}

\newtheorem*{*assump}{Assumption}
\newtheorem*{*hyp}{Temporary Hypothesis}

\DeclareMathOperator{\id}{id}

\newcommand{\bd}{\partial}

\newcommand{\sm}{\smallsetminus}

\newcommand{\GG}{\mathcal{G}}

\newcommand{\LL}{\mathcal{L}}

\newcommand{\Z}{\mathbb{Z}}

\newcommand{\D}{\mathbb{D}}

\renewcommand{\ss}{\subset}

\newcommand{\0}{\emptyset}

\newcommand{\wt}{\widetilde} 
\renewcommand{\tilde}{\wt}
\newcommand{\ol}{\overline} 
\newcommand{\wh}{\widehat}

\newcommand{\Si}{S^{1}_{\infty}} 
 
\newcommand{\Se}{S_{\infty}} 
 
\renewcommand{\tilde}{\wt}

\renewcommand{\phi}{\varphi}

\renewcommand{\epsilon}{\varepsilon}

\begin{document}

\title{Homeomorphisms of surfaces of finite type}

\author[J. Cantwell]{John Cantwell}
\address{Department of Mathematics\\ St. Louis University\\ St. 
Louis, MO 
63103}
\email{CANTWELLJC@SLU.EDU}




\begin{abstract}
We   give a proof  of the Neilsen-Thurston classification theorem of a homeomorphism $f$ of a standard surface of finite type  as either periodic, pseudo-Anosov, or reducible.  In the periodic case, we show that there exists an integer $n>0$ such that $f$ is isotopic to $h$ with $h^{n}$  isotopic to the identity. This is the weaker version of the Nielsen-Thurston theorem.
\end{abstract}

\maketitle

\section{Introduction}

A surface $L$ is a \emph{standard surface} if $L$ admits a complete hyperbolic metric which contains no hyperbolic  halfplanes and in which all the boundary components are geodesic~\cite[Definition~2]{cc:epstein}. A surface is of \emph{finite type} if it is of finite genus and has at most finitely many punctures and circle boundary components. 

Throughout this paper $L$ is a fixed, standard hyperbolic surface of finite type and $f:L\to L$ is a fixed homeomorphism. Thus,  $L$ is assumed to have a complete hyperbolic structure, contains no hyperbolic halfplanes,  has finite genus, and at most finitely many cusps and geodesic circle boundary components.   We allow the surface $L$ to be nonorientable and the homeomorphism $f$ to be orientation reversing.

Let $\Gamma'$ be the collection of simple closed  geodesics $\gamma$  in $L\sm\bd L$ such that $f_{*}^{n}(\gamma) = \gamma$ for some $n>0$ (see Definition~\ref{fstar} for the definition of $f_{*}$). Let $\Gamma$ consist of those elements of $\Gamma'$ that are isolated in the sense that they are disjoint from all other elemants of $\Gamma'$. By the definition of $\Gamma$ and since $L$ has finite genus, $\Gamma$ consists of finitely many disjoint geodesics.  Since $f_{*}(\Gamma') = \Gamma'$, it follows that $f_{*}(\Gamma) = \Gamma$.

We prove, 

\begin{theorem}\label{NTthm}

If $L$ be a connected, standard surface  of finite type and $f:L\to L$ a homeomorphism, then there exists a homeomorphism $h:L\to L$ isotopic to $f$ such that either, 

\begin{enumerate}

\item  $h^{n}$ is isotopic to the identity for some integer $n>0$\upn{;}\label{NT1}

\item $h$ permutes a pair of transverse geodesic lamination $\Lambda_{\pm}$ which itersect every closed two sided  geodesic in $L\sm\bd L$ and are minimal in the sense that every leaf of each lamination is dense in the lamination\upn{;}\label{NT2}

\item If $S$ is the internal completion \emph{(Definition~\ref{intcompl})} of a component $U$ of $L\sm|\Gamma|$, then there is a least integer $n_{S}>0$ such that $h^{n_{S}}(S) = S$.  Further, $h^{n_{S}}|S:S\to S$ satisfies \emph{(\ref{NT1})} or \emph{(\ref{NT2})}.\label{part3}

\end{enumerate}

\end{theorem}

If $L$ is orientable,  then this is~\cite[Theorem~0.2]{ha:th}. We do not prove~\cite[Theorem~0.1]{ha:th} whose proof  uses different  ideas. We concentrate on constructing the laminations $\Lambda_{\pm}$. Our proof borrows  from Handel and Thurston~\cite{ha:th} and Casson and Bleiler~\cite{bca} but is substantially different from either proof. 

\begin{rem}

If $S$ is the internal completion of a component $U$ of $L\sm|\Gamma|$, then $\bd S = S\sm U$  consists of geodesics. The inclusion map $U\to L$ extends in a natural way to an immersion $S\to L$  that sends these geodesics to geodesics in $\Gamma$. This immersion may send two geodesics in $\bd S$ to the same geodesic in $\Gamma$ or, if the image geodesic in $\Gamma$ is one sided, it will be the two-to-one image of a geodesic in $\bd S$. 

\end{rem}

\begin{rem}[continued]

In Theorem~\ref{NTthm}~(\ref{part3}) we abuse notation and denote the natural extension of $h|U$ to $S$ again by $h$.

\end{rem}

\begin{defn}\label{supp}

If $\GG$ is a set of  geodesics in $L$, we let $|\GG|$ denote the \emph{support} of $\GG$, that is the union of the elements of $\GG$.

\end{defn}

\begin{defn}

If $U$ is   a component of $L\sm|\Gamma|$ which contains a simple closed geodesic $\sigma$ such that the geodesics $f_{*}^{n}(\sigma)$, $n\in\Z$, are distinct, then $U$ is called a \emph{pseuso-Anosov component}. Otherwise $U$ is called a \emph{periodic component}.

\end{defn}

Here $f_{*}$ is the geodesic tightening map (Definition~\ref{fstar}).

\begin{rem}

Since $f_{*}$ permutes the elements of $\Gamma$  we can extend the domain of definition of $f_{*}$ to  the components of $L\sm|\Gamma|$ so that $f_{*}$ permute the components of $L\sm|\Gamma|$. 

\end{rem}

\begin{conv}

If $\gamma$ a geodesic in $\Delta$ and $\{\gamma_{n}\}_{n\ge0}$ a sequence of geodesics in $\Delta$ we  use the notation $\gamma_{n}\to\gamma$ as $n\to \infty$ for uniform convergence in the Euclidean metric on $\Delta$.

\end{conv}

\begin{defn}\label{defnLL}

Suppose $U$ is a pseudo-Anosov component  of $L\sm|\Gamma|$ and $\sigma\ss U\sm\bd L$ is a simple closed  geodesic. Define $\LL^{\sigma}_{+}$ (respectively $\LL^{\sigma}_{-}$) to be the set of simple geodesics $\gamma\ss U$ such that for every integer $p>0$,  there exists a subsequence $\{\sigma_{n_{k}p}\}_{k\ge 0}$ of $\{\sigma_{n}\}_{n\ge 0}$ (respectively $\{\sigma_{n}\}_{n\le 0}$) such that for any lift $\wt\gamma$ of $\gamma$ there exists a lift $\{\wt\sigma_{n_{k}p}\}_{k\ge 0}$ of $\{\sigma_{n_{k}p}\}_{k\ge 0}$ such that $\wt \sigma_{n_{k}p}\to\wt\gamma$ as $k\to\infty$. 

\end{defn}

\begin{rem}

The full strength of Definition~\ref{defnLL} is needed in the proof of Lemma~\ref{perleaf}.

\end{rem}

In Section~\ref{constrLambdapm} we prove, 

\begin{theorem}\label{proplams}

If $U$ is a pseudo-Anosov component  of $L\sm|\Gamma|$ and $\sigma\ss U\sm\bd L$ is a simple closed  geodesic, then $\sigma_{n} = f_{*}^{n}(\sigma)$, $n\in\Z$, are distinct, $\0\ne|\LL^{\sigma}_{\pm}|\ss U$,  and the geodesics in $\LL^{\sigma}_{\pm}$ are  homeomorphic to the reals.   Further,

\begin{enumerate}

\item There exists a unique minimal geodesic laminations $\Lambda_{\pm}\ss\LL^{\sigma}_{\pm}$\upn{;}

\item $\Lambda_{+}$ and $\Lambda_{-}$ are transverse\upn{;}

\item The components of $U\sm(|\Lambda_{+}|\cup|\Lambda_{-}|)$ consist  of open rectangles and finitely many open disks and/or open or half open annuli\upn{;}

\item The laminations $\Lambda_{\pm}$ are independent of the choice of $\sigma$\upn{;}

\item $\LL^{\sigma}_{\pm}\sm\Lambda_{\pm}$ are each finite sets\upn{;}


\end{enumerate}

\end{theorem}

\section{Definitions and preliminary results}\label{defres}

We take   the open unit disk $\Delta$ with the Poincar\'e metric as hyperbolic plane. The closed unit disk $\D^{2} = \Delta\cup\Si$, where the  unit circle $\Si$ is called the \emph{circle at infinity}.

Throughout, $L$ is a connected, standard surface  of finite type. We fix a lift $\wt L$ to $\Delta$ and let $\Se = \bd\wt L\ss\D^{2}$. We denote a lift of $f:L\to L$ by $\wt f:\wt L\to\wt L$.  We denote by $\wh L$ the closure of $\wt L$ in $\D^{2}$ and refer to $E = \wh L\cap\Si$ as the \emph{ideal boundary} of $L$. By~\cite[Theorem~2]{cc:epstein}, the map   $\wt f:\wt L\to\wt L$ has a natural extension $\wh f:\wh L\to\wh L$.

\subsection{Surfaces of finite type}\label{fintype}

The following propositions are well known result in hyperbolic geometry.

\begin{prop}\label{areahypsur}

If $L$ is a standard hyperbolic of finite type with $c\ge0$ cusps, $m\ge0$ cross caps, $b\ge 0$ circle boundary components, and $g\ge 0$ handles, then $${\rm area}\,(L) = 2\pi(c+m+b+2g-2).$$

\end{prop}

\begin{rem}

By~\cite[Theorem~8]{cc:epstein}, if $L$ is a standard hyperbolic surface,  then $$c+m+b+2g\ge3.$$

\end{rem}
 
\begin{prop}

Each cusp of $L$ has a neighborhood such that no simple closed geodesic in $L$ meets that neighborhood.\label{cuspnhb}

\end{prop}

\subsection{Pseudo-geodesics and geodesic tightening}

\begin{defn}[pseudo-geodesic]\label{pseudogeodesic}
A curve  $\gamma\ss L$ is a  \emph{pseudo-geodesic} if either some (hence every) lift $\wt\gamma$ has two distinct, well defined endpoints on $\Si$  or $\gamma$ is a properly embedded, boundary incompressible  compact arc.
\end{defn}

Remark that geodesics and  essential embedded circles in $L$ that do not bound cusps   are pseudo-geodesics and that if $f$ is continuous and $\gamma$ is a geodesic, then by~\cite[Theorem~2]{cc:epstein}  $f(\gamma)$ is a pseudo-grodesic.

\begin{defn}

If $\gamma$ is a pseudo-geodesic, then the \emph{geodesic tightening} of $\gamma$ is the unique geodesic whose lifts have the same endpoints on $\Se$ as the lifts of $\gamma$.

\end{defn}

\begin{defn}\label{fstar}

If $\gamma$ is a geodesic, we will demote the geodesic tightening of $f(\gamma)$ by $f_{*}(\gamma)$. If $\GG$ is a set of geodesics, we define
$$f_{*}(\GG) = \{f_{*}(\gamma)\ |\ \gamma\in\GG\}.$$

\end{defn}

The next lemma follow immediately from~\cite[Theorem~2]{cc:epstein}. 

\begin{lemma}\label{f*seq}

If $\gamma_{k}$, $k\ge0$, and $\gamma$ are geodesics with lifts $\wt\gamma_{k}$, $k\ge0$, and $\wt\gamma$ such that $\wt\gamma_{k}\to\wt\gamma$ as $k\to\infty$ and $\wt f_{*}$ is a lift of $f_{*}$, then $\wt f_{*}(\gamma_{k})\to\wt f_{*}(\wt\gamma)$ as $k\to\infty$.

\end{lemma}

\subsection{Laminations}

\begin{defn}

A \emph{lamination} $\Lambda$ is a set of simple geodesics in $L$ whose support $|\Lambda|$  is closed in $L$.  A geodesic in $\Lambda$ is called a \emph{leaf} of $\Lambda$.

\end{defn}

\begin{defn}

A nonempty lamination which properly contains no other lamination is said to be \emph{minimal}.

\end{defn}

\begin{defn}[semi-isolated]\label{defnsemiis}
If a leaf $\lambda\in\Lambda$ is approached by points of  $|\Lambda|$ on at most one side, we say $\lambda$ is \emph{semi-isolated}.
\end{defn}

\begin{rem}

Note that our definition of semi-isolated includes all isolated leaves.  

\end{rem}

\subsection{Internal completion}

Let $\Lambda$ be a lamination in $L$. A component $V$  of $L\sm|\Lambda|$ has two different hyperbolic metrics, the first defined by geodesics of $L$ and the second defined by segments of  geodesics of $L$ contained in $V$.

\begin{defn}\label{intcompl}

Define the \emph{internal completion} of a component $V$ of $L\sm|\Lambda|$ to be the metric completion of $V$ in the metric defined  by segments of  geodesics of $L$ contained in $V$.

\end{defn}

\begin{defn}

If $V$ is a component of $L\sm|\Lambda|$, define a \emph{border component} of $v$  to be a boundary component of the internal completion of $V$. Let $\delta V$ denote the set of border compoents of $V$.    The set $|\delta V|$ is called the \emph{border} of $V$.

\end{defn}

\begin{lemma}

The point $x\in|\delta V|$ if and only if there exists a segment $[x,y]$ of a geodesic of $L$ such that $(x,y]\ss V$ and $x\notin V$.

\end{lemma}

There is a natural map from the internal completion of a component $V$ of $L\sm|\Lambda|$ into $L$ taking $\delta V$ to a set of semi-isolated leaves of $\Lambda$. An isolated leaf of $\lambda$ might be the image of two border components of $V$ under this map but  a semi-isolated leaf of $\Lambda$ which is not isolated is the image of at most one.

\begin{rem}

In abuse of notation we will often identify $\delta V$ with its image under this map and consider $\delta V$ to be a subset of $\Lambda$.

\end{rem}

\begin{rem}

For more information on border and internal completion see~\cite[Section~5.1]{cc:hm}.

\end{rem}

\section{Simple closed geodesics in $L$}\label{simplegeo}

Let $\sigma$ be a simple closed geodesic in $L$. Note that since $\sigma$ is a geodesic, $\sigma$ is essential and can not bound a cusp.  Throughout this paper, we use the notation,

\begin{nota}

 $\sigma_{n} = f_{*}^{n}(\sigma)$, $n\in\Z$.
 
 \end{nota}

\begin{lemma}

If there exists $n\in\Z$ and $k>0$ such that $\sigma_{n} = \sigma_{n+k}$,
then $\sigma_{m} = \sigma_{m+k}$, all $m\in\Z$.

\end{lemma}

\begin{proof}

$\sigma_{m} = f_{*}^{m-n}\circ f_{*}^{n}(\sigma) = f_{*}^{m-n}\circ f_{*}^{n+k}(\sigma) = \sigma_{m+k}.$
\end{proof}

Thus, one of the following two possibilities holds,

\begin{enumerate}

\item $\sigma_{k} = \sigma$ some least integer $k>0$;

\item $\sigma_{n}$, $n\in\Z$, are distinct.

\end{enumerate}

The following lemma is clear.

\begin{lemma}

If $\gamma$ is a simple closed geodesic and $\tau$ is any simple geodesic, then $\tau$ spirals on $\gamma$ if and only if there are lifts $\wt\gamma$ of $\gamma$ and $\wt\tau$ of $\tau$ sharing an endpoint on $\Si$.

\end{lemma}

\begin{prop}\label{spirals}

If $\gamma$ is a simple closed geodesic, then there exists $\epsilon>0$ such that if  $\tau$ is a simple geodesic $\tau$ containing a point at a distance less than $\epsilon$ from $\sigma$, then either,

\begin{enumerate}

\item $\tau$ spirals on $\gamma$ or\upn{;}

\item $\gamma$ and $\tau$ intersect transversely.

\end{enumerate}

\end{prop}

\begin{proof}

Compare~ the proof of \cite[Lemma~4.6]{bca}.   Choose a lift $\wt\gamma$ of $\gamma$. Let   $T_{\wt\gamma}$ be an orientation preserving deck transformation with $\wt\gamma$ as axis and $d$ the hyperbolic distance a point of $\wt\gamma$ is moved by $T_{\wt\gamma}$. Choose $\epsilon>0$ such that if $\sigma$ is any geodesic in $\Delta$ disjoint from $\wt\gamma$, not sharing an endpoint on $\Si$ with $\wt\gamma$, and     containing a point at a distance less than $\epsilon$ from $\wt\gamma$, then the perpendicular projection (see~\cite[Figure~4.5]{bca})  of $\sigma$ on $\wt\gamma$ has length greater than $d$.

Suppose that $\tau$ is a simple geodesic that does not spiral on $\gamma$, does not intersect $\gamma$, and does contain a point at a distance  less than $\epsilon$ from $\gamma$. Then $\tau$ has a lift $\wt\tau$ containing a point at a distance less than $\epsilon$ from $\wt\gamma$, does not intersect $\wt\gamma$, and does not share an endpoint on $\Si$ with $\wt\gamma$. It follows that the perpendicular projection of $\wt\tau$ on $\wt\gamma$ has length greater than $d$. Thus, $T_{\wt\gamma}(\wt\tau)$ meets $\wt\tau$ contradicting the fact that $\tau$ is simple. It follows that either $\tau$ intersects $\gamma$ or that $\tau$ spirals on $\gamma$.

The proof of the converse is immediate.
\end{proof}

For the rest of Section~\ref{simplegeo} we make the following assumption.

\begin{*assump}\label{assumeone}

Assume that,

\begin{enumerate}

\item $\gamma$ and $\sigma$ are simple closed geodesics\upn{;}

\item There exist a subsequence $\{\sigma_{n_{k}}\}_{k\ge0}$ of $\{\sigma_{n}\}_{n\ge0}$ and lifts $\wt\gamma$ of $\gamma$ and $\wt\sigma_{n_{k}}$ of $\sigma_{n_{k}}$ such that $\wt\sigma_{n_{k}}\to\wt\gamma$ as $k\to\infty$. 

\end{enumerate}

\end{*assump}

By Proposition~\ref{spirals}, we can assume that all the $\sigma_{n_{k}}$ intersect $\gamma$ transversely,

\begin{cor}\label{spiral}

If $\gamma$ and $\sigma$ satisfy  the above \emph{Assumption}, then there exists a subsequence $\{\sigma_{n_{\ell}}\}_{\ell\ge0}$ of  $\{\sigma_{n_{k}}\}_{k\ge0}$ and a geodesic $\tau$ such that each $\sigma_{n_{\ell}}$ has lift $\wt\sigma'_{n_{\ell}}$ with $\wt\sigma'_{n_{\ell}}\to\wt\tau$, a lift of $\tau$, and $\tau$ spirals on $\gamma$.

\end{cor}

\begin{proof}

Denote by $T_{\wt\gamma}$ an orientation preserving deck transformation with axis $\wt\gamma$ and let $[a,b]\ss\Si$ be such that $T_{\wt\gamma}(a) = b$. Each $\sigma_{n_{k}}$ has a lift meeting $\wt\gamma$ with an endpoint in $[a,b)$ so there exists a subsequence $\{\sigma_{n_{\ell}}\}_{\ell\ge0}$ of $\{\sigma_{n_{k}}\}_{k\ge0}$ and a geodesic $\tau$ such that each $\sigma_{n_{\ell}}$ has lift $\wt\sigma'_{n_{\ell}}$ meeting $\wt\gamma$ with an endpoint in $[a,b)$ and $\wt\sigma'_{n_{\ell}}\to\wt\tau$, a lift of $\tau$ with endpoint in $[a,b]$.

The geodesics $\tau$ and $\gamma$ can not intersect since each $\sigma_{n_{\ell}}$ is simple and there exist sequences of lifts of the $\sigma_{n_{\ell}}$ approaching both $\wt\tau$ and $\wt\gamma$.  Since each $\wt\sigma'_{n_{\ell}}$ meets $\wt\gamma$ and $\wt\sigma'_{n_{\ell}}\to\wt\tau$, it follows that $\wt\gamma$ and $\wt\tau$ share an endpoint on $\Si$. Thus $\tau$ spirals on $\gamma$.
\end{proof}

\begin{nota}

The notation $\#(A)$ denotes the number of points in the finite set $A$.

\end{nota}

\begin{cor}

If $\gamma$ and $\sigma$ satisfy  the above \emph{Assumption}, then  for each integer $r>0$, $f^{r}_{*}(\gamma)$ and $\gamma$ are either disjoint or coincide.

\end{cor}

\begin{proof}

Let $N_{r} = \#(\sigma\cap  \sigma_{r}) = \#(\sigma_{n_{\ell}}\cap \sigma_{r+n_{\ell}})$ and $\tau$ be the geodesic constructed in Corollary~\ref{spiral}.  If $\gamma$ and $f_{*}^{r}(\gamma)$ intersect transversly, then $\tau$ and $f_{*}^{r}(\tau)$ have infinitely many points of transverse intersection.  Choose disjoint neighborhoods $V_{j}$, $1\le j\le N_{r}+1$, of $N_{r}+1$ points of transverse intersetion of $\tau$ and $f_{*}^{r}(\tau)$. By taking $\ell$ sufficiently large one can find a point of $\sigma_{n_{\ell}}\cap \sigma_{r+n_{\ell}}$ in each $V_{j}$. Since there are $N_{r}$ points in $\sigma_{n_{\ell}}\cap \sigma_{r+n_{\ell}}$ and $N_{r}+1$ of the $V_{j}$ this is a contradiction and the lemma follows.
\end{proof}

The next corollary follows since $L$ has finite genus.

\begin{cor}\label{inGamma}

If $\gamma$ and $\sigma$ satisfy  the above \emph{Assumption},  then for some least integer $r>0$, $f^{r}_{*}(\gamma) = \gamma$, $f_{*}^{r}$ fixes  the orientation of $\gamma$, and $f^{r}$ is orientation preserving.

\end{cor}

Let $c,d$ be the endpoints of $\wt\gamma$ and $\wt f_{*}^{r}$ be a lift of $f_{*}^{r}$ fixing $\wt\gamma$ and thus both $c$ and $d$.

\begin{cor}\label{transto}

If $\gamma$ and $\sigma$ satisfy  the above \emph{Assumption} and  $\alpha$ is any simple closed geodesic transverse to $\gamma$, then $f^{n}_{*}(\alpha)$, are distinct, $n\in\Z$.

\end{cor}

\begin{proof}

Suppose $f^{n}_{*}(\alpha) = \alpha$ for some $n>0$   and       choose the integer $r$ and lifts $\wt\gamma$ and $\wt f_{*}^{r}$ as above . By choosing $r$ larger, we can further assume that $\wt f^{r}_{*}(\wt\alpha) = \wt\alpha$ for every lift $\wt\alpha$ of $\alpha$ meeting $\wt\gamma$. By Proposition~\ref{spirals}, we can assume that $\wt\sigma_{n_{k}}\cap\wt\gamma\ne0$ all $k\ge 1$. Let $T_{\wt\gamma}$ be aa orientation preserving deck transformation with axis $\wt\gamma$. Choose $x\in\wt\gamma$ and let $y = T_{\wt\gamma}(x)$. Let $\GG$ be the finite set of lifts of the geodesics $\sigma_{0},\ldots,\sigma_{r-1}$ meeting the interval $[x,y]\ss\wt\gamma$. Since $\GG$ is finite, there exists lifts $\wt\alpha$ and $\wt\alpha'$ of $\alpha$ meeting $\wt\gamma$ such that all the geodesic in $\GG$ lie between $\wt\alpha$ and $\wt\alpha'$ in $\Delta$. 

Consider $\sigma_{n}$ where $n = i + jr$ for integers $0\le i\le r-1$ and $j\ge 0$. Then, $\sigma_{n} = f_{*}^{jr}(\sigma_{i})$. If $\beta$ is a lift of $\sigma_{n}$ that meets $\wt\gamma$, let $\rho =  \wt  f_{*}^{-jr}(\beta)$. Then $\rho$ is a lift of $\sigma_{i}$  that meets $\wt\gamma$. Thus, there exists an integer $s\in\Z$, such that $\rho = T_{\wt\gamma}^{s}(\zeta)$ with $\zeta\in\GG$ a lift of $\sigma_{i}$ meeting the interval $[x,y]\ss\wt\gamma$. Since $\zeta$ is between  the lifts $\wt\alpha$ and $\wt\alpha'$   of $\alpha$, it follows that $\rho$ is between   the lifts $T_{\wt\gamma}^{s}(\wt\alpha)$ and $T_{\wt\gamma}^{s}(\wt\alpha')$  of $\alpha$. Thus, $\beta = \wt  f_{*}^{jr}(\rho)$ is between $\wt  f_{*}^{jr}\circ T_{\wt\gamma}^{s}(\wt\alpha) = T_{\wt\gamma}^{s}(\wt\alpha)$ and $\wt  f_{*}^{jr}\circ T_{\wt\gamma}^{s}(\wt\alpha') = T_{\wt\gamma}^{s}(\wt\alpha')$. That is,  any lift of $\sigma_{n}$, $n\in\Z$, meeting   $\wt\gamma$ lies  between $T_{\wt\gamma}^{s}(\wt\alpha)$ and $T_{\wt\gamma}^{s}(\wt\alpha')$, for some $s\in\Z$. 

Therefore, given $k\ge 0$, there exsts an integer $s_{k}\in\Z$ such that $\wt\sigma_{n_{k}}$ lies between $T_{\wt\gamma}^{s_{k}}(\wt\alpha)$ and $T_{\wt\gamma}^{s_{k}}(\wt\alpha')$, some $s_{k}\in\Z$. Thus, the sequence $\{\wt\sigma_{n_{k}}\}_{k\ge0}$ does not converge to $\wt\gamma$. This contradicts our Assumption and the  corollary follows.
\end{proof}

\section{The  laminations $\Lambda^{\sigma}_{\pm}$}\label{constrLambdapm}

\begin{*assump}\label{assumetwo}

Throughout Section~\ref{constrLambdapm},   $U$ is a fixed component of $L\sm|\Gamma|$ and $\sigma\ss U$ is a fixed geodesic such that the geodesics $\sigma_{n} = f_{*}^{n}(\sigma)$, $n\in\Z$, are distinct.

\end{*assump}

\subsection{$\LL^{\sigma}_{+}$ and the positive  lamination $\Lambda^{\sigma}_{+}$}\label{poslam}

In this and the next section we explicitly prove results for $\LL^{\sigma}_{+}$ and the positive lamination $\Lambda^{\sigma}_{+}$. The same results hold for $\LL^{\sigma}_{-}$ and the negative lamination $\Lambda^{\sigma}_{-}$ with the obvious modifications.

\begin{nota}

Let $n_{U}$ denote the least positive integer such that $f_{*}^{n_{U}}(U) = U$.

\end{nota}

\begin{lemma}

$\LL^{\sigma}_{+}\ne\0$.

\end{lemma}

\begin{proof}
By Proposition~\ref{cuspnhb} there exists a compact subsurface $L'\ss L$   such that every simple closed geodesic in $L$ is contained in $L'$. Define, 
$$X^{\sigma}_{+}(p) = \bigcap_{n=1}^{\infty}\ol{\bigcup_{k=n}^{\infty}f_{*}^{p!n_{U}k}(\sigma)}\ss L'\cap U.$$
Note that,
$$X^{\sigma}_{+}(1)\supset X^{\sigma}_{+}(2)\supset\cdots\supset X^{\sigma}_{+}(p)\supset\cdots.$$
Define,
$$X^{\sigma}_{+} = \bigcap_{p=1}^{\infty}X^{\sigma}_{+}(p)\ss L'\cap U.$$
The set $X^{\sigma}_{+}\ne\0$ by the compactness of $L'$. We now show that $X^{\sigma}_{+}\ss|\LL^{\sigma}_{+}|$.

Suppose $x\in X^{\sigma}_{+}$ and $p>0$ is an integer. We will show that there exists a simple geodesic $\beta\in\LL^{\sigma}_{+}$ containing $x$.

Since $x\in X^{\sigma}_{+}$, there exists a subsequence $\{\sigma_{n_{k}p}\}_{k\ge 0}$ of $\{\sigma_{p!n_{U}n}\}_{n\ge 0}$ and points $x_{k}\in\sigma_{n_{k}p}$ such that $x_{k}\to x$. Let $v_{k}$ be the unit tangent vector to $\sigma_{n_{k}p}$ at $x_{k}$. By again passing to a subsequence if neccesary we can assume that $v_{k}\to v$, a unit tangent vector to $L$ at $x$. Let $\beta$ be the unique geodesic through $x$ with tangent vector $v$ and $\wt\beta$ any lift of $\beta$. Let $\wt x$ be the lift of $x$ in $\wt\beta$ and choose lifts $\wt x_{k}$ of $x_{k}$ with $\wt x_{k} \to \wt x$. Choose lifts $\wt\sigma_{n_{k}p}$ of $\sigma_{n_{k}p}$ containing $\wt x_{k}$, $k\ge 0$.   Then $\wt \sigma_{n_{k}p}\to \wt\beta$ so $\beta\in\LL^{\sigma}_{+}$. 

If $\beta$ is not simple, then $\beta$ has lifts $\wt\beta$ and $\wt\beta'$ such that $\wt\beta\cap\wt\beta'\ne\0$. Then for $k$ sufficiently large, $\sigma_{n_{k}p}$ has lifts $\wt\sigma_{n_{k}p}$ and $\wt\sigma_{n_{k}p}'$ such that $\wt\sigma_{n_{k}p}\cap\wt\sigma_{n_{k}p}'\ne\0$ so $\sigma_{n_{k}p}$ is not simple. This contradiction implies that $\beta$ is simple.

Thus, $x\in\beta\in\LL^{\sigma}_{+}$ so $x\in|\LL^{\sigma}_{+}|$.
\end{proof}

\begin{prop}\label{homeoR}

Every geodesic in $\LL^{\sigma}_{+}$ is homeomorphis to the reals.

\end{prop}

\begin{proof}

If  $\gamma\in\LL^{\sigma}_{+}$, then  there exists a lift $\wt\gamma$ of $\gamma$ and  a sequence $\{\sigma_{n_{k}}\}_{k\ge0}$ with lift $\{\wt\sigma_{n_{k}}\}_{k\ge0}$  such that $\wt \sigma_{n_{k}}\to\wt\gamma$ as $k\to\infty$. If $\gamma$ is closed, by Proposition~\ref{spirals}, we can assume that the $\sigma_{n_{k}}$ meet $\gamma$ transversely. Corollary~\ref{inGamma} then implies that $\gamma\in\Gamma'$. By Corollary~\ref{transto} it follows that $\gamma\in\Gamma$. Since $\sigma\ss U$, it follows that $\sigma$ and each $\sigma_{n}$, $n\in\Z$, are disjoint from each geodesic in $\Gamma$ and thus from $\gamma$. This contradiction implies the proposition.
\end{proof}

Since $\delta U$ consists of closed geodesics we have,

\begin{cor}\label{ssU}

$|\LL^{\sigma}_{+}|\ss U$.

\end{cor}

Let $\beta\in\LL^{\sigma}_{+}$. Recall that the closure of a geodesic is the disjoint union of geodesics.  If $\Lambda_{\beta}$ denotes the lamination consisting of the geodesic in $\ol\beta$, then,

\begin{lemma}\label{convto}

 $\Lambda_{\beta}\ss\LL^{\sigma}_{+}$.

\end{lemma}

\begin{proof}

Suppose $\gamma\in\Lambda_{\beta}$, $\wt\gamma$ is a lift of $\gamma$,  and $p>0$ is an integer.  Since $\beta\in\LL^{\sigma}_{+}$, there exists a subsequence $\{\sigma_{n_{k}p}\}_{k\ge0}$ of $\{\sigma_{n}\}_{n\ge0}$ such that if $\wt\beta$ is a lift of $\beta$ there exists a lift $\{\wt\sigma_{n_{k}p}\}_{k\ge0}$  of $\{\sigma_{n_{k}p}\}_{k\ge0}$ such that $\wt\sigma_{n_{k}p}\to\wt\beta$ as $k\to\infty$. Fix the sequence $\{\sigma_{n_{k}p}\}_{k\ge0}$  but not the sequence $\{\wt\sigma_{n_{k}p}\}_{k\ge0}$  of lifts.

Since $\gamma\ss\ol\beta$, we can choose lifts $\wt\beta_{\ell}$ of $\beta$ such that each endpoint on $\Si$ of $\wt\beta_{\ell}$ is at a distance less than $1/\ell$ from the corresponding endpoint of $\wt\gamma$. Inductively, suppose integers $k_{1}<\cdots<k_{\ell-1}$ and lfts $\wt\sigma_{n_{k_{i}}p}$ of $\sigma_{n_{k_{i}}p}$ 
have been chosen so that the corresponding endpoints on $\Si$ of $\wt\beta_{i}$ and $\wt\sigma_{n_{k_{i}}p}$ are at a distance less than $1/i$, $1\le i\le\ell-1$. Since each $\wt\beta_{\ell}$ is a lift of $\beta$, there exists  an integer $k_{\ell}>k_{\ell}-1$ and lift $\wt\sigma_{n_{k_{\ell}}p}$ of $\sigma_{n_{k_{\ell}}p}$ such that   the corresponding endpoints on $\Si$ of $\wt\beta_{\ell}$ and $\wt\sigma_{n_{k_{\ell}}p}$ are at a distance less than $1/\ell$. Then $\wt\sigma_{n_{k_{\ell}}p}\to\wt\gamma$ as $\ell\to\infty$. Let $\wt\gamma'$ be any lift of $\gamma$. Then there exists a deck transformation $T$ so that $\wt\gamma' = T(\wt\gamma)$. Then $\{T(\wt\sigma_{n_{k_{\ell}}p})\}_{\ell\ge0}$ is a   lift of the sequence $\{\sigma_{n_{k_{\ell}}p}\}_{\ell\ge0}$ and $T(\wt\sigma_{n_{k_{\ell}}p})\to T(\wt\gamma) = \wt\gamma'$ as $\ell\to\infty$ so $\gamma\in\LL^{\sigma}_{+}$ as desired.
\end{proof}

By compactness of $L'$ we have,

\begin{cor}

There exist minimal laminations that are subsets of $\LL^{\sigma}_{+}$.

\end{cor}

\begin{defn}

Let  $\Lambda^{\sigma}_{+}$ be a fixed minimal lamination which is a subset of $\LL^{\sigma}_{+}$.  

\end{defn}

We show in Proposition~\ref{lamuniq} that the choice of $\Lambda^{\sigma}_{+}$ is unique.

\begin{lemma}

The space  $|\Lambda^{\sigma}_{+}|$ has empty interior. The geodesics in $\Lambda^{\sigma}_{+}$  are the path components of $|\Lambda^{\sigma}_{+}|$. Further, $|\Lambda^{\sigma}_{+}| = \ol{\lambda}$ for any $\lambda\in\Lambda^{\sigma}_{+}$. 

\end{lemma}

\begin{proof}

If $|\Lambda^{\sigma}_{+}|$ has interior, then $|\Lambda^{\sigma}_{+}|$ is a surface without boundary  because the closure of any boundary component would be a  proper sublamination of $\Lambda^{\sigma}_{+}$ contradicting the fact that $\Lambda^{\sigma}_{+}$ is minimal. Thus, $|\Lambda^{\sigma}_{+}| = L$ is a foliated surface with negative Euler characteristic which is a contradiction. Thus, $|\Lambda^{\sigma}_{+}|$ has empty interior. It follows that the path componets of $|\Lambda^{\sigma}_{+}|$ are the geodesics in $\Lambda^{\sigma}_{+}$. That $|\Lambda^{\sigma}_{+}| = \ol{\lambda}$ for any $\lambda\in\Lambda^{\sigma}_{+}$ follows since $\Lambda^{\sigma}_{+}$ is minimal.
\end{proof}

The next corollary follows since $\Lambda^{\sigma}_{+}$ is a minimal lamination containing no  closed geodesic.

\begin{cor}

 $\Lambda^{\sigma}_{+}$ is transversely a Cantor set.

\end{cor}

\begin{lemma}

Suppose $\wt \sigma_{n_{k}}\to\wt\lambda$ where $\ol\lambda = |\Lambda^{\sigma}_{+}|$.   Then $\wt \sigma_{n_{k}+r}\to\wt\lambda_{r} = \wt{f_{*}^{r}(\lambda)}$ and $\ol\lambda_{r} = |f_{*}^{r}(\Lambda^{\sigma}_{+})|$.

\end{lemma}

\begin{lemma}\label{disjlams}

$\Lambda^{\sigma}_{+}$ and $f_{*}^{r}(\Lambda^{\sigma}_{+})$ have no transverse points of intersection.  Since $\Lambda^{\sigma}_{+}$ is minimal it follows that $\Lambda^{\sigma}_{+}$ and $f_{*}^{r}(\Lambda^{\sigma}_{+})$ are either disjoint or coincide.

\end{lemma}

\begin{proof}

Let $N_{r} = \#(\sigma\cap  \sigma_{r}) = \#(\sigma_{n_{k}}\cap \sigma_{r+n_{k}})$. If $\Lambda^{\sigma}_{+}$ and $f_{*}^{r}(\Lambda^{\sigma}_{+})$ have one point of transverse intersection, they have infinitely many points of transverse intersection.  Choose disjoint neighborhoods $V_{j}$, $1\le j\le N_{r}+1$, of $N_{r}+1$ points of transverse intersetion of $\Lambda^{\sigma}_{+}$ and $f_{*}^{r}(\Lambda^{\sigma}_{+})$. By taking $k$ sufficiently large one can find a point of $\sigma_{n_{k}}\cap \sigma_{r+n_{k}}$ in each $V_{j}$. Since there are $N_{r}$ points in $\sigma_{n_{k}}\cap \sigma_{r+n_{k}}$ and $N_{r}+1$ of the $V_{j}$ this is a contradiction and the lemma follows.
\end{proof}

\subsection{Crown sets}\label{crsets}

If $V$ is a  component of $L\sm|\Lambda^{\sigma}_{+}|$, then a geodesic $\gamma^{+}_{0}\in\delta V$ determines a bi-infinite sequence $\{\gamma^{+}_{i}\}_{i\in\Z}\ss\delta V$ with lift $\{\wt\gamma^{+}_{i}\}_{i\in\Z}$ such that the terminal end of $\wt\gamma^{+}_{i}$ and initial end of $\wt\gamma^{+}_{i+1}$  share the endpoint $a_{i}$ on $\Si$.

\begin{lemma}

There are finitely many components $V$ of $L\sm|\Lambda^{\sigma}_{+}|$  and for each such $V$, $\delta V$ is a finite set.

\end{lemma}

\begin{proof}

Let $V$ be a component of $L\sm|\Lambda^{\sigma}_{+}|$,  $X$ be the internal completion of $V$, and  $DX$ denote the double of $X$. Then,  by~\cite[Lemma 1]{cc:epstein},  $DX$ is a standard hyperbolic surface. If $\delta V$ is an infinite set, then $DX$ has infintely many cusps and thus infinite area contradicting the fact that $L$ is of finite type.  By Proposition~\ref{areahypsur} and the remark in Section~\ref{fintype}, ${\rm area}\,(DX)\ge2\pi$ so ${\rm area}\,(X)\ge\pi$. Since ${\rm area}\,(L)$ is finite, it follows that there are finitely many components of the complement of $|\Lambda^{\sigma}_{+}|$.
\end{proof}

Let $V$ be a component of $L\sm|\Lambda^{\sigma}_{+}|$ and $\{\gamma^{+}_{i}\}_{i\in\Z}\ss\delta V$  with lift $\{\wt\gamma^{+}_{i}\}_{i\in\Z}$  be as above. Let $\wt V$ be the lift of $V$ with $\{\wt\gamma^{+}_{i}\}_{i\in\Z}\ss\delta\wt V$.  Then the covering projection $\wt V\to V$ maps  the closure of the convex hull of $\{\wt\gamma^{+}_{i}\}_{i\in\Z}$ onto a set $C_{+}\ss V$.

\begin{defn}

The set $C_{+}$    is called the \emph{crown set associated} to $\{\gamma^{+}_{i}\}_{i\in\Z}$. 

\end{defn}

\begin{lemma}

A crown set is either homeomorphic to,

\begin{enumerate}

\item A disk with $p$ points removed from the  boundary\upn{;}

\item A punctured disk with $p$ points removed from the  boundary\upn{;}

\item An annulus with $p$ points removed from one boundary\upn{;}

\item A M\"obius strip with $p$ points removed from its boundary.

\end{enumerate}

\end{lemma}

\begin{proof}

If $\{\wt\gamma^{+}_{i}\}_{i\in\Z}$ is a finite set as in Figure~\ref{CSleft}, then $C_{+}$ is a disk with $p$ points removed from the  boundary. Otherwise, $\{\wt\gamma^{+}_{i}\}_{i\in\Z}$ is an infinite set. The points $a_{i}$, $i\in\Z$, limit on two points in $\Si$. If these two points coincide as in Figure~\ref{CScusp}, then $C_{+}$ is a punctured disk with $p$ points removed from the  boundary. If these two points are distinct as in Figure~\ref{CSright}, then $C_{+}$ is either an annulus with $p$ points removed from one boundary or a M\"obius strip with $p$ points removed from its boundary depending on whether the geodesic $\wt\rho\ss\Delta$ with endpoints the two limit points projects under the covering projection to a two sided or one sided simple closed curve  $\rho\ss L$.
\end{proof}

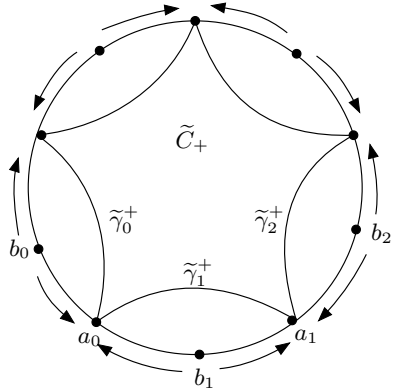
\begin{figure}[t]
\begin{center}
\begin{picture}(200,140)(-20,-140)
\rotatebox{270}{\includegraphics[width=140pt]{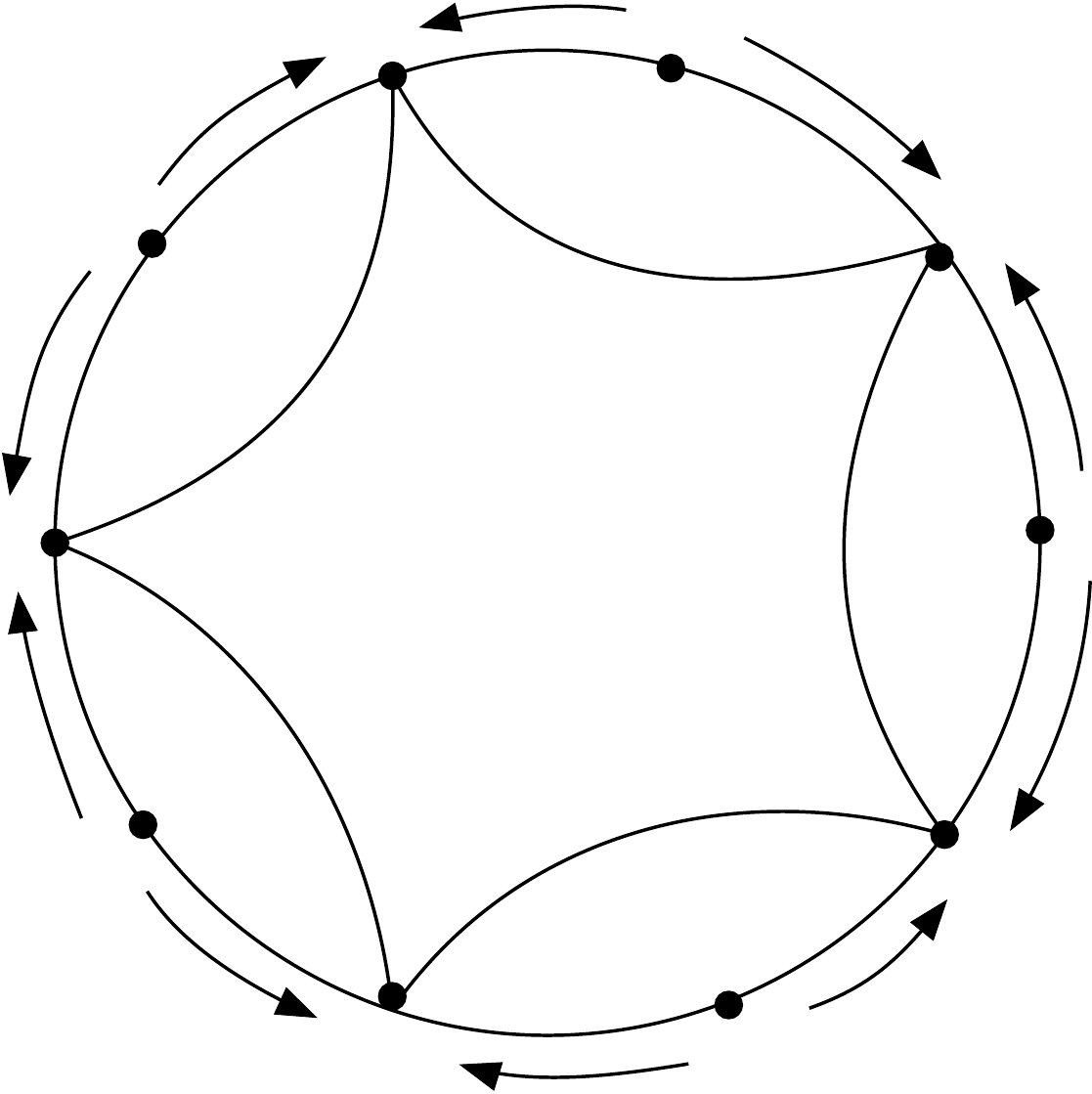}}

\put (-77,-55){\small$\wt C_{+}$}
\put (-102,-85){\small$\tilde{\gamma}^{+}_{0}$}
\put (-47,-85){\small$\tilde{\gamma}^{+}_{2}$}
\put (-74,-105){\small$\tilde{\gamma}^{+}_{1}$}
\put (-140,-96){\small$b_{0}$}
\put (-3,-90){\small$b_{2}$}
\put (-70,-145){\small$b_{1}$}
\put (-32,-128){\small$a_{1}$}
\put (-114,-129){\small$a_{0}$}

\end{picture}
\caption{Lift of simply connected crown set $C_{+}$}\label{CSleft}
\end{center}
\end{figure}

\begin{defn}

If the crown set $C_{+}$ is an annulus with $p$ points removed from one boundary, then $C_{+}$ is called a \emph{crown}~\cite[Figure~4.2]{bca} and the simple closed curve boundary component $\rho$ of $C_{+}$ is called the \emph{rim} of the crown $C_{+}$.

\end{defn}

\begin{rem}

In the case that $C_{+}$ is a M\"obius strip with $p$ points removed from its boundary, the geodesic  $\wt\rho\ss\Delta$ projects under the covering projection to the center geodesic of the M\"obius strip.

\end{rem}

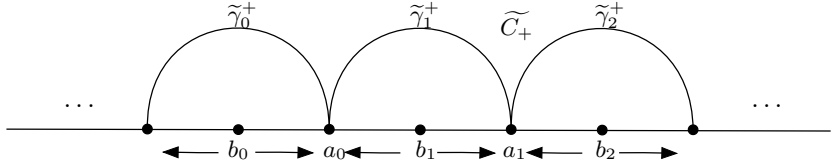
\begin{figure}[h]
\begin{center}
\begin{picture}(180,40)(60,-40)
\rotatebox{270}{\includegraphics[width=50pt]{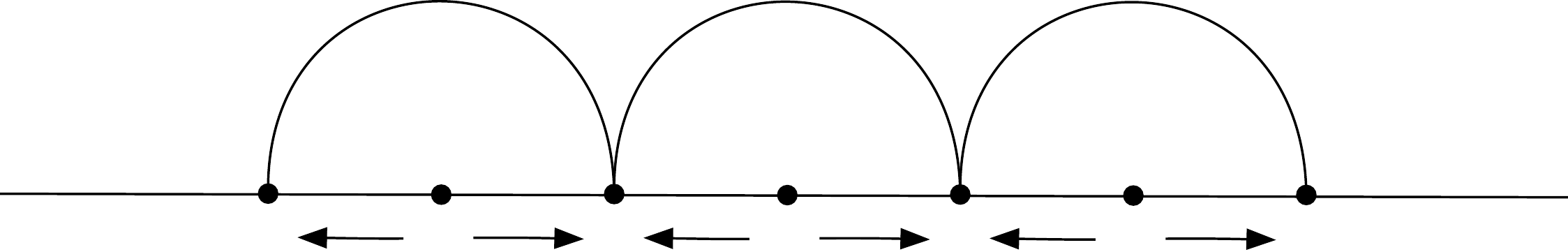}}

\put (-125,-3){\small$\tilde{C_{+}}$}
\put (-159,2){\small$\tilde{\gamma}^{+}_{1}$}
\put (-228,2){\small$\tilde{\gamma}^{+}_{0}$}
\put (-89,2){\small$\tilde{\gamma}^{+}_{2}$}
\put (-228,-49){\small$b_{0}$}
\put (-89,-49){\small$b_{2}$}
\put (-157,-49){\small$b_{1}$}
\put (-192,-49){\small$a_{0}$}
\put (-124,-49){\small$a_{1}$}
\put (-30,-30){$\dots$}
\put (-290,-30){$\dots$}

\end{picture}
\caption{Lift of crown set $C_{+}$ containing a cusp}\label{CScusp}
\end{center}
\end{figure}

Let $R^{\sigma}$ be the set consisting of all rims of  crowns of $\Lambda^{\sigma}_{+}$ and all center geodesics of   M\"obius strips that are crown sets.    Let $U^{\sigma}$ be the   component of $L\sm |R^{\sigma}|$ containing $\Lambda^{\sigma}_{+}$.  Then $R^{\sigma}_{k} = f_{*}^{k}(R^{\sigma})$ is the set consisting of all rims of crowns of $f_{*}^{k}(\Lambda^{\sigma}_{+})$ and all center geodesics of   M\"obius strips that are crown sets of $f_{*}^{k}(\Lambda^{\sigma}_{+})$. Let $U^{\sigma}_{k}$ be the component of $L\sm |R^{\sigma}_{k}|$ containing $f_{*}^{k}(\Lambda^{\sigma}_{+})$.

\begin{figure}[b]
\begin{center}
\begin{picture}(180,100)(60,-100)
\rotatebox{270}{\includegraphics[width=100pt]{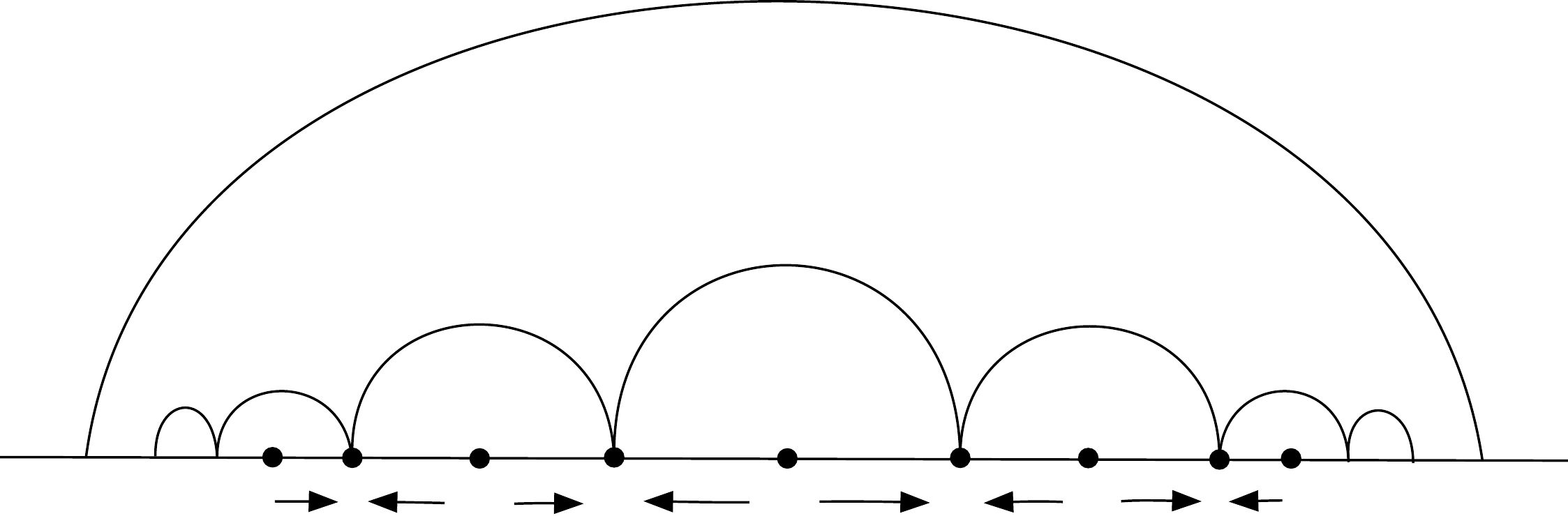}}

\put (-153,-49){\small$\tilde{\gamma}^{+}_{1}$}
\put (-215,-60){\small$\tilde{\gamma}^{+}_{0}$}
\put (-96,-60){\small$\tilde{\gamma}^{+}_{2}$}
\put (-58,-20){\small$\tilde{\rho}$}
\put (-215,-99){\small$b_{0}$}
\put (-96,-99){\small$b_{2}$}
\put (-153,-99){\small$b_{1}$}
\put (-187,-99){\small$a_{0}$}
\put (-121,-99){\small$a_{1}$}
\put (-120,-29){\small$\tilde{C_{+}}$}
\put (-30,-85){$\dots$}
\put (-286,-85){$\dots$}

\end{picture}
\caption{Lift of crown $C_{+}$ with rim $\rho$}\label{CSright}
\end{center}
\end{figure}
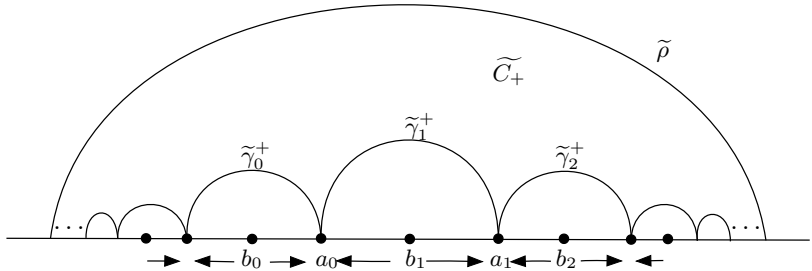

\begin{lemma}\label{minr}

For some minimal $r\ge 0$, $U^{\sigma} = U^{\sigma}_{r}$ and $\Lambda^{\sigma}_{+} = f_{*}^{r}(\Lambda^{\sigma}_{+})$.

\end{lemma}

\begin{proof}

If no crown set of $\Lambda^{\sigma}_{+}$ is a crown, then $L\sm\Lambda^{\sigma}_{+}$ consists of crown sets, so $\Lambda^{\sigma}_{+}\cap f_{*}(\Lambda^{\sigma}_{+})\ne\0$. By Lemma~\ref{disjlams}, $\Lambda^{\sigma}_{+} = f_{*}(\Lambda^{\sigma}_{+})$.

Otherwise, $\Lambda^{\sigma}_{+}$ has at least one crown. First note that if an element of $R^{\sigma}_{i}$ meets an element  of $R^{\sigma}_{j}$ transversely, it follows that  $f_{*}^{i}(\Lambda^{\sigma}_{+})$ meets  $f_{*}^{j}(\Lambda^{\sigma}_{+})$ transversely which contradicts Lemma~\ref{disjlams}. Thus, an element of $R^{\sigma}_{i}$ either coincides with or is disjoint from an element of $R^{\sigma}_{j}$. It follows that the set  $\bigcup_{j=0}^{\infty}R^{\sigma}_{j}$  is finite and there is a smallest $r\ge 0$ so that  $R^{\sigma}_{r }= R^{\sigma}$ and $U^{\sigma}_{r }= U^{\sigma}$. Thus, $\Lambda^{\sigma}_{+}\cap f_{*}^{r}(\Lambda^{\sigma}_{+})\ne\0$. Again by Lemma~\ref{disjlams}, $\Lambda^{\sigma}_{+} = f_{*}^{r}(\Lambda^{\sigma}_{+})$.
\end{proof}

Let $\lambda\in\Lambda^{\sigma}_{+}$ have lift $\wt\lambda$ with endpoints $a_{0}, a_{1}\in\Si$ as in Figure~\ref{b1b2}.  Let $[a_{0}, a_{1}]$ be one of the two intervals on $\Se$ with endpoints $a_{0}, a_{1}$. Let $p>0$ be an integer such that $f_{*}^{p}(\lambda) = \lambda$, $f_{*}^{p}$ preserves the orientation of $\lambda$, and $f^{p}$ is orientation preserving on $L$. Of necessity $p$ is a multiple of $r$. Let $\wt f_{*}^{p}$ be a lift of $f_{*}^{p}$ such that $\wt f_{*}^{p}(\wt\lambda) = \wt\lambda$.

\begin{figure}[h]
\begin{center}
\begin{picture}(200,140)(80,-140)
\rotatebox{270}{\includegraphics[width=145pt]{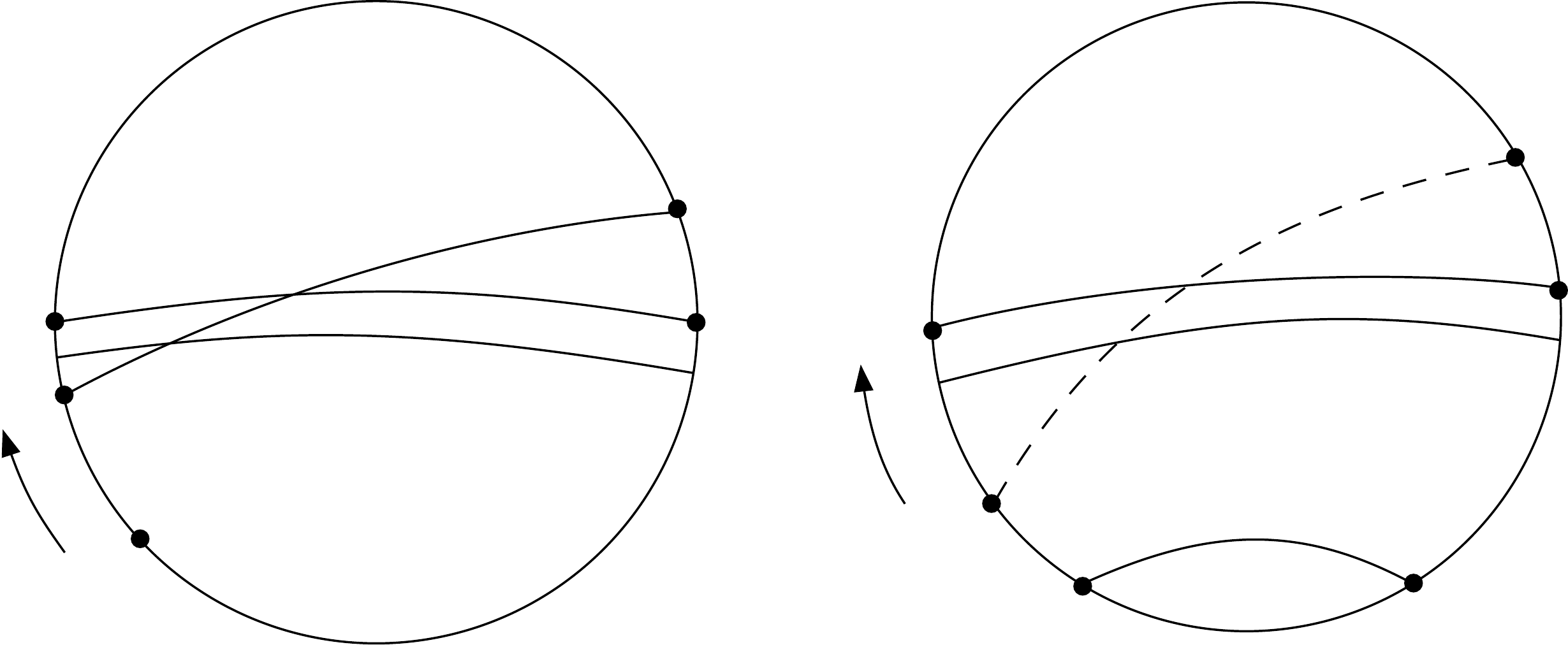}}

\put (-315,-65){\small$\tilde{\lambda}$}
\put (-250,-50){\small$\zeta$}
\put (-265,-87){\small$\tilde{\lambda}_{0}$}
\put (-355,-72){\small$a_{0}$}
\put (-193,-76){\small$a_{1}$}
\put(-335,-127){\small$A_{0}$}
\put(-350,-94){\small$A$}
\put (-196,-48){\small$B$}
\put(-143,-119){\small$A_{0}$}
\put (-8,-37){\small$B_{0}$}
\put (-175,-76){\small$A=a_{0}$}
\put (3,-67){\small$a_{1}$}
\put (-118,-138){\small$b_{1}$}
\put (-30,-138){\small$b'_{1}$}
\put (-120,-63){\small$\tilde{\lambda}$}
\put (-77,-85){\small$\tilde{\lambda}_{1}$}
\put (-77,-117){\small$\tilde{\lambda}_{2}$}
\put (-60,-45){\small$\tilde{\sigma}$}

\end{picture}
\caption{Lemma~\ref{perleaf}}\label{b1b2}
\end{center}
\end{figure}

\begin{lemma}\label{perleaf}

If $\wt\lambda$ as above is approached by lifts of geodesics in $\Lambda^{\sigma}_{+}$ having endpoints in $(a_{0}, a_{1})\cap\Si$, then,

\begin{enumerate}

\item The points $a_{0}, a_{1}$ are contracting fixed points of $\wh f_{*}^{p}$ in the interval $[a_{0}, a_{1}]$\upn{;}\label{partone}

\item $\wh f_{*}^{p}$ has exactly one fixed point $b_{1}\in(a_{0}, a_{1})\ss\Se$ which is expanding. Further $b_{1}\in\Si$.\label{parttwo} 

\end{enumerate}

\end{lemma}

\begin{proof}

Any closed geodesic meeting $U^{\sigma}$ intersects $\Lambda^{\sigma}_{+}$ transversely. Thus $\sigma$ intersects $\lambda$ transversely.  Let $\wt\sigma$ be a lift of $\sigma$ with endpoints $A_{0},B_{0}\in\Si$ with $A_{0}\in(a_{0},a_{1})$ and $B_{0}\notin[a_{0},a_{1}]$. Thus $\wt\sigma\cap\wt\lambda\ne\0$. Let $A_{k} = \wh f_{*}^{kp}(A_{0})$ and $B_{k} = \wh f_{*}^{kp}(B_{0})$. Since $\wh f_{*}^{p}$ is a homeomorphism, it follows that the sequence $\{A_{k}\}_{k\ge 0}$ is monotone in $(a_{0}, a_{1})$ and that $A_{k}\to A\in[a_{0}, a_{1}]\cap\Si$. Similarly $B_{k}\to B\in\Si\sm(a_{0},a_{1})$.  Let  $\zeta$ be the geodesic in $\Delta$ with  endpoints $A,B$. Then $\wh f_{*}^{jp}(\wt\sigma)\to\zeta$ as $j\to\infty$..

Refer to Figure~\ref{b1b2}~(left). If $a_{0}\ne A\ne a_{1}$, by assumption, there exists a lift $\wt\lambda_{0}$ of a leaf of $\Lambda^{\sigma}_{+}$ with one endpoint in the interval $(a_{0},A)\cap\Si$ and the other endpoint in the interval $(A,a_{1})\cap\Si$. Since $\Lambda^{\sigma}_{+}\ss\LL^{\sigma}_{+}$, there exist a subsequence $\{\sigma_{n_{k}p}\}_{k\ge0}$ of the sequence $\{\sigma_{n}\}_{n\ge0}$ with lift  $\{\wt\sigma_{n_{k}p}\}_{k\ge0}$ such that $\wt\sigma_{n_{k}p}\to\wt\lambda_{0}$ as $k\to\infty$. Thus, for arbitarily large values of the integer $k$,  $\wt\sigma_{n_{k}p}$ has   one endpoint in the interval $(a_{0},A)\cap\Si$ and the other endpoint in the interval $(A,a_{1})\cap\Si$. Thus, for arbitrarily large integers $k$, $\wt\sigma_{n_{k}p}\cap\wt f_{*}^{n_{k}p}(\wt\sigma)\ne\0$. Since $\wt\sigma_{n_{k}p}$ and $\wt f_{*}^{n_{k}p}(\wt\sigma)$ are both lifts of $\sigma_{n_{k}p}$, it follows that $\sigma_{n_{k}p}$ is not simple. This contradiction implies that  $A=a_{0}$ or $A=a_{1}$. Without loss assume that $A=a_{0}$. The case $A=a_{1}$ is similar. From the above discussion, the point $a_{0}$ is a contracting fixed point of $\wh f_{*}^{p}$ in the interval $[a_{0}, a_{1}]\ss\Se$ and the only fixed point in the interval $[a_{0}, A_{0})$.

Refer to Figure~\ref{b1b2}~(right).  Choose $\lambda_{1}\in\Lambda^{\sigma}_{+}$ with lift $\wt\lambda_{1}$ having one endpoint in $(a_{0},A_{0})\cap\Si$. Then, $\wt f_{*}^{np}(\wt\lambda_{1})\to \wt\lambda$ as $n\to\infty$  so $a_{1}$ is also a contracting fixed point  of $\wh f_{*}^{p}$ in the interval $[a_{0}, a_{1}]$.
  
If the endpoints of $\wt f_{*}^{-np}(\wt\lambda_{1})$ converge to $b_{1}\ne b'_{1}\in(a_{0},a_{1})\cap\Si$ as $n\to\infty$, then $b_{1}$ and $b'_{1}$ are both contracting fixed points of $\wh f_{*}^{-p}$ and thus expanding fixed points of $\wh f_{*}^{p}$ in the subinterval $[b_{1},b'_{1}]\ss\Se$ which contains the points $a_{0}, a_{1}$. On the other hand, since $p$ is a multiple of $r$, by Lemma~\ref{minr} it follows that the $\wt f_{*}^{-np}(\wt\lambda_{1})$ are lifts of leaves of $\wt\Lambda^{\sigma}_{+}$ which converge to the lift $\wt\lambda_{2}$ of a leaf of $\Lambda^{\sigma}_{+}$ with endpoints $b_{1},b'_{1}\in\Si$.  Part (\ref{partone}) of this lemma applied to this interval $[b_{1},b'_{1}]$ and the geodesic $\wt\lambda_{2}$ implies that $b_{1}$ and $b'_{1}$ are both  contacting fixed points of $\wh f_{*}^{p}$ in the interval $[b_{1},b'_{1}]$. This contradiction implies $b_{1}=b'_{1}$ and (\ref{parttwo}) follows.
\end{proof}

\begin{figure}[h]
\begin{center}
\begin{picture}(200,140)(50,-140)
\rotatebox{270}{\includegraphics[width=140pt]{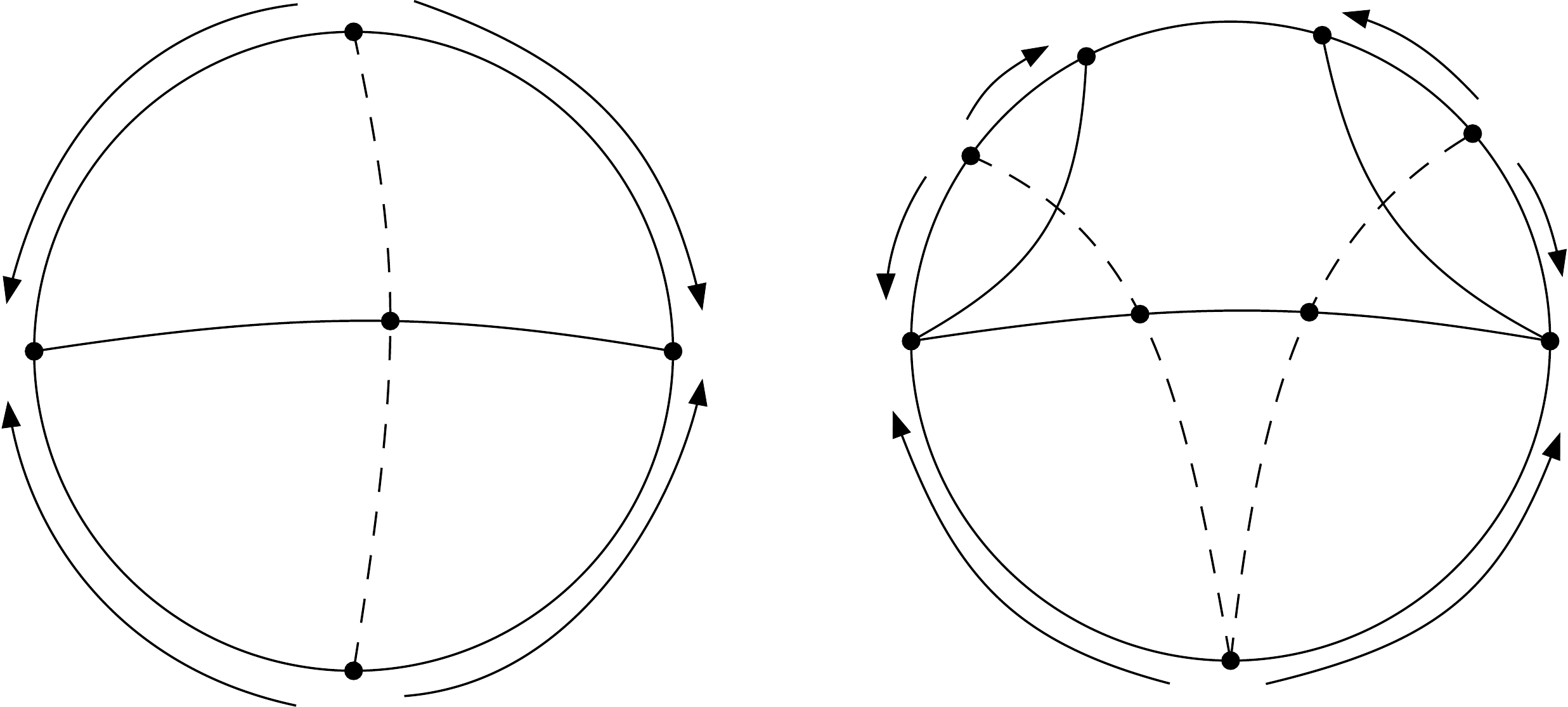}}

\put (-280,-62){\small$\tilde{\lambda}$}
\put (-96,-35){\small$\tilde{\gamma}^{+}_{0}$}
\put (-55,-35){\small$\tilde{\gamma}^{+}_{2}$}
\put (-81,-58){\small$\tilde{\lambda} = \tilde{\gamma}^{+}_{1}$}
\put (-75,-20){\small$\wt C_{+}$}
\put (-132,-30){\small$b_{0}$}
\put (-15,-28){\small$b_{2}$}
\put (-245,-2){\small$b_{0}$}
\put (-243,-144){\small$b_{1}$}
\put (-69,-142){\small$b_{1}$}
\put (-317,-70){\small$a_{0}$}
\put (-173,-70){\small$a_{1}$}
\put (-143,-68){\small$a_{0}$}
\put (0,-70){\small$a_{1}$}
\put (-229,-42){\small$\wt\gamma$}
\put (-87,-95){\small$\wt\gamma^{-}_{1}$}
\put (-58,-95){\small$\wt\gamma^{-}_{2}$}

\end{picture}
\caption{Two Cases}\label{TwoPics}
\end{center}

\end{figure}
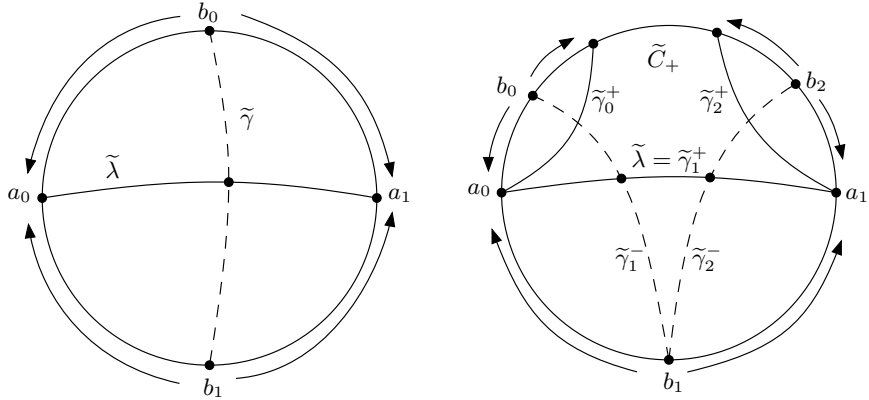

Further,

\begin{prop}\label{bothsides}

If $\wt\lambda$ is as in \emph{Lemma~\ref{perleaf}}, then either,

\begin{enumerate}

\item $\wt\lambda$ is appoached from above and below by leaves of $\wt\Lambda^{\sigma}_{+}$ and there exist exactly four fixed points $a_{0},a_{1},b_{0},b_{1}$ of $\wh f_{*}^{p}$ in $\Se$  as in \emph{Figure~\ref{TwoPics} (left)}\upn{;}\label{prtone}

\item $\wt\lambda$ is bordered above by the lift of a crown set $C_{+}$ as in \emph{Figure~\ref{TwoPics} (right)} and $\lambda = \gamma^{+}_{1}$ belongs to a cycle $\{\gamma^{+}_{i}\}_{i\in\Z}\ss\delta C_{+}$ with lifts   $\{\wt\gamma^{+}_{i}\}_{i\in\Z}$  such that, for $i\in\Z$,  \label{prttwo}

\begin{enumerate}

\item $\wt\gamma^{+}_{i}$ and $\wt\gamma^{+}_{i+1}$ share endpoint $a_{i}\in\Si$\upn{;}

\item $a_{i}$ is a contracting fixed point of $\wh f_{*}^{p}$ in $\Se$\upn{;}\label{part2}

\item There exists $b_{i}\in(a_{i-1},a_{i})\cap\Si$ which is an expanding fixed point of $\wh f_{*}^{p}$ in $\Se$ and the only fixed point of $\wh f_{*}^{p}$ in $(a_{i-1},a_{i})$.

\end{enumerate}

\end{enumerate}

\end{prop}

\begin{proof}

If $\wt\lambda$ is approached both above and below (in Figure~\ref{b1b2}~(right)) by lifts of leaves of $\Lambda^{\sigma}_{+}$, then (\ref{prtone}) follows by applying Lemma~\ref{perleaf}~(\ref{partone}) to each of the intervals $[a_{0},a_{1}]\ss\Se$.

Otherwise, there exists a component $V$ of $L\sm|\Lambda^{\sigma}_{+}|$ such that $\lambda\in\delta V$. Thus, there exists a crown set $C_{+}$ and cycle of geodesics $\{\gamma^{+}_{i}\}_{i\in\Z}\ss\delta C_{+}$ such that  $\lambda = \gamma^{+}_{1}$ belongs to the cycle. The lift $\wt\lambda$ of $\lambda$ fixed by $\wt f_{*}^{p}$ determines lifts $\wt\gamma^{+}_{i}$ of $\gamma^{+}_{i}$, $i\in\Z$, all fixed by $\wt f_{*}^{p}$. Then (\ref{prttwo})  follows from applying Lemma~\ref{perleaf}~(\ref{parttwo}) to each of the geodesics $\wt\gamma^{+}_{i}$, $i\in\Z$.
\end{proof}

\begin{rem}

We will see that the pairs of points  $b_{i-1},b_{i}\in\Si$, $i\in\Z$, are each endpoints of leaves of $\wt\Lambda^{\sigma}_{-}$

\end{rem}

\begin{example}

Proposition~\ref{bothsides}~(\ref{prttwo}) is illustrated in Figure~\ref{CSleft} if the crown set $C_{+}$ is simply connected, in Figure~\ref{CScusp} if the crown set is a cusp, and Figure~\ref{CSright} if the crown set is an annulus or a M\"obius strip.

\end{example}

\begin{cor}\label{trtau}

If $\tau$ is a simple closed geodesic transverse to $\Lambda^{\sigma}_{+}$, then $\tau_{n} = f_{*}^{n}(\tau)$, $n\in\Z$, are distinct.

\end{cor}

\begin{proof}

Let $\wt\tau$ be a lift of $\tau$ meeting $\wt\lambda$ at a point close to $a_{0}\in\Si$. If $f_{*}^{k}(\tau) = \tau$, then $\wt f_{*}^{-nkp}(\wt\tau)$, $n\ge 0$, is a sequence of lifts of $\tau$ and as such should not accumulate in $\Delta$. One can see from Figure~\ref{b1b2} that  this sequence  limits on the geodesic with endpoints $b_{0},b_{1}\in\Si$. The corollary follows from this contradiction.
\end{proof} 

\begin{cor}

If $\tau$ is a simple closed geodesic  meeting $|R^{\sigma}|$, then $\tau_{n} = f_{*}^{n}(\tau)$, $n\in\Z$, are distinct.

\end{cor}

\begin{proof}

If $\tau$ meets $|R^{\sigma}|$, then $\tau$ is transverse to $\Lambda^{\sigma}_{+}$ and the corollary follows from the previous corollary.
\end{proof}

The previous two corollaries have as immediate consequence.

\begin{cor}\label{corcomp}\hfill

\begin{enumerate}

\item The elements of $R^{\sigma}$ not components of $\bd L$ are contained in $\Gamma$\upn{;}

\item If $U^{\sigma}$ is the component  of $L\sm|R^{\sigma}|$ containing $\Lambda^{\sigma}_{+}$, then $U^{\sigma} = U$\upn{;}

\item $U\sm\Lambda^{\sigma}_{+}$ is a finite union of sets whose closures are crown sets\upn{;}\label{itemcompl}

\item If $\tau\ss U\sm\bd L$ is a simple closed  geodesic, then   $\tau_{n} = f_{*}^{n}(\tau)$, $n\in\Z$, are distinct.

\end{enumerate}

\end{cor}

\subsection{The negative  lamination $\Lambda^{\sigma}_{-}$}\label{neglams}

 We define the negative lamination $\Lambda^{\sigma}_{-}$ and develop its properties exactly as in Sections~\ref{poslam} and~\ref{crsets}.  All the results of Sections~\ref{poslam} and~\ref{crsets}  hold for   $\Lambda^{\sigma}_{-}$   with the obvious modifications. 

Proposition~\ref{neglam} below relates $\Lambda^{\sigma}_{-}$ and $\Lambda^{\sigma}_{+}$. We need two preliminary lemmas.

\begin{lemma}\label{pmtrans}

If $\lambda_{-}\in\Lambda^{\sigma}_{-}$ and $\lambda_{+}\in\Lambda^{\sigma}_{+}$ are  semi-isolated leaves, then $\lambda_{-}\cap\lambda_{+}$ is nonempty and transverse.

\end{lemma}

\begin{proof}

Since the components of $U\sm\Lambda_{+}$  are open disks, open or half-open annuli, or open M\"obius strips, it follows that $\lambda_{-}\cap\lambda_{+}$ is nonempty. It remains to show that $\lambda_{-}\ne\lambda_{+}$. Suppose $\lambda_{-} = \lambda_{+} = \lambda$ and choose the integer $p>0$ and lifts $\wt\lambda$ of $\lambda$ and $\wt f_{*}^{p}$ of $f_{*}^{p}$  so that $\wt f_{*}^{ p}(\wt\lambda) = \wt\lambda$ and $\wt f_{*}^{p}$ preserves the orientation of $\wt L$ and $\wt\lambda$. Refer to Figure~\ref{TwoPics}~(right). Then Proposition~\ref{bothsides}~(\ref{prttwo}) applied to $\Lambda^{\sigma}_{+}$ implies that the endpoints $a_{0},a_{1}\in\Si$ of $\wt\lambda$ are contracting fixed points of $\wt f_{*}^{p}$ while Proposition~\ref{bothsides}~(\ref{prttwo}) applied to $\Lambda^{\sigma}_{-}$ implies that the endpoints $a_{0},a_{1}\in\Si$ of $\wt\lambda$ are expanding fixed points of $\wt f_{*}^{p}$. This contradiciton implies $\lambda_{-}\ne\lambda_{+}$. 
\end{proof}

\begin{lemma}\label{bothends}

If  $\lambda\in\Lambda^{\sigma}_{\pm}$  and $\gamma$ is a geodesic meeting $\lambda$ transversely, then the points $\lambda\cap\gamma$ accumulate at   both ends of $\lambda$.

\end{lemma}

\begin{proof}

The lemma follows immediately since  $\Lambda^{\sigma}_{\pm}$ are minimal laminations.
\end{proof}

\begin{prop}\label{neglam}

\begin{enumerate}

\item In \emph{Proposition~\ref{bothsides}~(\ref{prtone})} applied to $\wt\Lambda^{\sigma}_{+}$, there exists a minimal lamination $\Lambda\ss\LL^{\sigma}_{-}$ such that the geodesic $\wt\gamma$ with endpoints $b_{0},b_{1}\in\Si$, is the lift of  a   leaf  $\gamma\in\Lambda$ that is not semi-isolated in $\Lambda$\upn{;}\label{itemone}

\item In \emph{Proposition~\ref{bothsides}~(\ref{prttwo})} applied to $\wt\Lambda^{\sigma}_{+}$, there exists a minimal lamination $\Lambda\ss\LL^{\sigma}_{-}$ such that  the geodesics $\wt\gamma^{-}_{i}$ with endpoints $b_{i-1},b_{i}\in\Si$, are the lifts of  semi-isolated leaves   $\gamma^{-}_{i}\in\Lambda$ $i\in\Z$.\label{itemtwo}

\end{enumerate}

\end{prop}

\begin{proof}

(\ref{itemone}) Refer to  Figure~\ref{TwoPics}~(left). Let $\lambda\in\Lambda^{\sigma}_{+}$ with lift $\wt\lambda\in\wt\Lambda^{\sigma}_{+}$ be the leaf of Proposition~\ref{bothsides}~(\ref{prtone}).  Denote by $\wt\gamma$ the geodesic in $\Delta$    having endpoints $b_{0},b_{1}\in\Si$.  

By Lemma~\ref{bothends}, lifts of $\sigma$ approach both ends of $\wt\lambda$ so there exists a lift $\wt\sigma$ of $\sigma$ meeting $\wt\lambda$ to the left of $\wt\gamma$ in Figure~\ref{TwoPics}~(left).  Thus, $\wt\sigma_{kp} = \wt f_{*}^{kp}(\wt\sigma)\to\wt\gamma$ from the left as $k\to-\infty$. It follows that $\wt\gamma$ is the lift of a simple geodesic $\gamma\ss L$.  Let $q>0$ be an integer and let $n_{k} = -kp$. Since $\{\wt\sigma_{n_{k}q}\}_{k\ge0}$ is a subsequence of $\{\wt\sigma_{kp}\}_{k\le0}$, it follows that $\wt\sigma_{n_{k}q}\to\wt\gamma$ as $k\to\infty$. By Definition~\ref{defnLL},  $\gamma\in\LL^{\sigma}_{-}$. 

Let $\Lambda$ be the lamination consisting of the geodesics in $\ol\gamma$. By Lemma~\ref{bothends}, if $\gamma'\in\Lambda$, then there exists a lift $\wt\gamma'$ of $\gamma'$ meeting $\wt\lambda$ to the left of $\wt\gamma$ in Figure~\ref{TwoPics}~(left). Thus, $\wt f_{*}^{kp}(\wt\gamma')\to\gamma$ from the left as $k\to-\infty$. Similarly, there exists a lift $\wt\gamma''$ of $\gamma'$ such that $\wt f_{*}^{kp}(\wt\gamma''_{-})\to\wt\gamma$ from the right as $k\to-\infty$. It follows that $\gamma$ is not semi-isolated in $\Lambda$ and that $\Lambda$ is a minimal lamination. By the previous paragraph $\Lambda\ss\LL^{\sigma}_{-}$  and (\ref{itemone}) is proven. 

(\ref{itemtwo}) is proven similarly. Refer to Figures~\ref{TwoPics}~(right), \ref{CSleft}, \ref{CScusp}, and~\ref{CSright}.
\end{proof}

We take $\Lambda$ of Proposition~\ref{neglam}~(\ref{itemtwo}) as $\Lambda^{\sigma}_{-}$.

\begin{cor}\label{trlams}

$\Lambda^{\sigma}_{+}$ and $\Lambda^{\sigma}_{-}$ are transverse.

\end{cor}

\begin{proof}

Refer to Figures~\ref{TwoPics}~(right). The geodesic $\wt\gamma^{-}_{1}$ with endpoints $b_{0},b_{1}\in\Si$ is a semi-isolated leaf of $\wt\Lambda^{\sigma}_{-}$ which is transverse to $\wt\Lambda^{\sigma}_{+}$. Since $\Lambda^{\sigma}_{-}$ is minimal, if $\gamma^{-}_{1}$ is the leaf of $\Lambda^{\sigma}_{-}$ with lift $\wt\gamma^{-}_{1}$, then $\gamma^{-}_{1}$ is dense in  $\Lambda^{\sigma}_{-}$. It follows that $\Lambda^{\sigma}_{-}$ is transverse to $\Lambda^{\sigma}_{+}$.
\end{proof}

Suppose $C_{+}$ is the  crown set of Proposition~\ref{bothsides}~(\ref{prttwo}) associated to $\{\gamma^{+}_{i}\}_{i\in\Z}\ss\Lambda^{\sigma}_{+}$  Let $p>0$, $\wt f_{*}^{p}$ be a lift of $f_{*}^{p}$, $\{\wt\gamma^{+}_{i}\}_{i\in\Z}$ be a lift of $\{\gamma^{+}_{i}\}_{i\in\Z}$, and $b_{i}$, $i\in\Z$, be fixed points of $\wt f_{*}^{p}$ as in Proposition~\ref{bothsides}~(\ref{prttwo}). Then, if $\gamma^{-}_{i}$ is the geodesic in $L$ with lift $\wt\gamma^{-}_{i}$ having endpoints $b_{i-1},b_{i}\in\Si$,  $i\in\Z$, then let $C_{-}$ be the crown set associated to $\{\wt\gamma^{-}_{i}\}_{i\in\Z}\ss\Lambda^{\sigma}_{-}$.

\begin{defn}

The crown sets $C_{\pm}$ of $\Lambda_{\pm}$ are said to be  \emph{complementary}. 

\end{defn}

The closure of components of $U\sm(|\Lambda^{\sigma}_{+}|\cup|\Lambda^{\sigma}_{-}|)$ are components of the intersection of (not necessarily complementary) crown sets  $C_{\pm}$ of $\Lambda^{\sigma}_{\pm}$. All but finitely many of these components  are rectangles.  The nonrectangular components consist of one component of the intersection of some pair of  complementary crown sets.

\begin{defn}

The nonrectangular component of $C_{+}\cap C_{-}$ is called the \emph{nucleus} of the complementary crown sets $C_{\pm}$. 

\end{defn}

\begin{rem}

The nulcleus of complemtary crown sets $C_{\pm}$ has one boundary component $\zeta$ which is piecewise geodesic closed curve with edges altenately in  $\gamma^{+}_{i}$ and $\gamma^{-}_{i+1}$, $i\in\Z$. Additionally the nucleus of complementary crowns $C_{\pm}$ will have a boundary component that is the common rim of $C_{\pm}$.

\end{rem}

\begin{cor}\label{compcompl}

The finitely many nuclei  of complementary crown sets $C_{\pm}$ of $\Lambda^{\sigma}_{\pm}$   are either,

\begin{enumerate}

\item Closed disks\upn{;}

\item Punctured disks\upn{;}

\item Annuli with one edge the common rim of the complementary crowns  $C_{\pm}$\upn{;}

\item M\"obius strips.

\end{enumerate}

\end{cor}

\subsection{Further properties of $\Lambda^{\sigma}_{\pm}$ and $\LL^{\sigma}_{+}$}

\begin{prop}\label{lamuniq}

The laminations $\Lambda^{\sigma}_{\pm}$ are independent of the choice of  minimal lamination $\Lambda^{\sigma}_{\pm}\ss\LL^{\sigma}_{\pm}$, and simple closed geodesic $\sigma\ss U$ such that $f_{*}^{n}(\sigma)$, $n\in\Z$, are distinct.

\end{prop}

\begin{proof}

We prove the proposition for $\Lambda^{\sigma}_{+}$. The other case is analogous. Suppose $\Lambda\ss\LL^{\tau}_{+}$ is a minimal lamination of   $\LL^{\tau}_{+}$ and $\tau\ss U$ a simple closed geodesic   such that $f_{*}^{n}(\tau)$, $n\in\Z$, are distinct. If $\gamma\in\Lambda$, since $\Lambda^{\sigma}_{+}$ is minimal and $\Lambda\ss U$, then $\gamma\cap\gamma^{-}_{1}$ clusters at both ends of $\gamma^{-}_{1}$ where $\gamma^{-}_{1}$ is the leaf of $\Lambda^{\sigma}_{-}$ with lift $\wt\gamma^{-}_{1}$ having endpoints  $b_{0},b_{1}\in\Si$ (Figure~\ref{TwoPics}~(right)). Choose a lift $\wt\gamma$  of $\gamma$ near $b_{1}$. Then $\wt f_{*}^{np}(\wt\gamma)\to\wt\gamma^{+}_{1}$ as $n\to\infty$. Thus $\gamma^{+}_{1}\in\Lambda$ so $\Lambda^{\sigma}_{+}\ss\Lambda$. Since $\Lambda$ is minimal, it follows that $\Lambda^{\sigma}_{+} = \Lambda$.
\end{proof}

\begin{defn}

Define $\Lambda_{\pm} = \Lambda^{\sigma}_{\pm}$.

\end{defn}

\begin{rem}

We will still use the  notation $\Lambda^{\sigma}_{\pm}$ because it specifies the geodesic $\sigma$.

\end{rem}

\begin{lemma}\label{finitemany}

$\LL^{\sigma}_{+}\sm\Lambda^{\sigma}_{+}$ and $\LL^{\sigma}_{-}\sm\Lambda^{\sigma}_{-}$ are each a finite set of geodesics.

\end{lemma}

\begin{proof}

We prove that $\LL^{\sigma}_{+}\sm\Lambda^{\sigma}_{+}$ is finite. The proof of the other case is analogous. We use the notation of Proposition~\ref{bothsides}~(\ref{prttwo}). Thus, if $\gamma\in\LL^{\sigma}_{+}\sm\Lambda^{\sigma}_{+}$, then $\gamma$ lies in a crown set $C_{+}$. It follows that the lift $\wt\gamma$ of $\gamma$ to $\wt C_{+}$ has vertices $a_{i}, a_{j}\in\Si$, $i\ne j\in\Z$. Since the set $\{\gamma^{+}_{i}\}_{i\in\Z}$ is finite, there can be only finitely many such $\gamma$.
\end{proof}







The proof of Theorem~\ref{proplams} is complete.

\section{Defining $h$ and the proof of Theorem~\ref{NTthm}}

By~\cite[Theorems~2.1 and~3.3]{Epstein:isotopy}  and induction,  there exists an ambient isotopy $\Phi:L\times[0,1]\to L$ such that $\Phi^{0} = \id$ and $\Phi^{1}\circ f(\gamma) = f_{*}(\gamma)$ for every   $\gamma\in\Gamma$. Here we use the notation, $\Phi^{t}(x) = \Phi(x,t)$. In abuse of notation, denote $\Phi^{1}\circ f$ by $f$.  Thus, from now on we  assume  that $f(\gamma) = f_{*}(\gamma)$  for every   $\gamma\in\Gamma$.














\begin{prop}

There exists a homeomorphism $h:L\to L$, isotopic to $f$such that,

\begin{enumerate}

\item If $U$ is  a component of $L\sm|\Gamma|$, then $h(U)$ is  a component of $L\sm|\Gamma|$\upn{;}\label{111} 

\item If $U$ is a periodic (respectively pseudo-Anosov) component, then $h(U)$  is a periodic (respectively pseudo-Anosov) component\upn{;}\label{222}

\item $h(\Lambda^{\sigma}_{\pm}) = \Lambda^{f_{*}(\sigma)}_{\pm}$ if $\sigma$ is a simple closed  geodesic in a pseudo-Anosov component.

\end{enumerate}

\end{prop}

\begin{proof}

Choose lifts $\wt L$ af $L$ and $\wt f:\ol L\to\wt L$ of $f:L\to L$. We will define $\wt h:\wt L\to\wt L$ so that $\wh h|\Se.= \wh f|\Se$.  First define $\wt h|\wt\gamma = \wt f|\wt\gamma$ if $\gamma\in\Gamma$ or $\gamma\ss\bd L$. Then (\ref{111}) follows since $f(\gamma) = f_{*}(\gamma)$  for every   $\gamma\in\Gamma$. Further,  if $U$ is a pseudo-Anosov somponent and $\sigma\ss U\sm\bd L$ is a simple closed  geodesic, then Lemma~\ref{f*seq} implies that  $f_{*}(\Lambda^{\sigma}_{\pm}) = \Lambda^{f_{*}(\sigma)}_{\pm}$ and $f(U)$ is a pseudo-Anosov component     and  (\ref{222}) follows. 

Next define  $\wt h|\wt U = \wt f|\wt U$ if $U$ is a periodic component. 

Now suppose $U$ is a pseudo-Anosov component and $\sigma\ss U$ is a simple closed  geodesic. To define $\wt h:\wt U\to\wt{h(U)}$, we mimic the proof of Lemma~6.1 of Casson-Bleiler~\cite[pp.~89-90]{bca}.   Choose  lifts $\wt U$ of $U$, $\wt{h(U)}$ of $h(U)$, and  $\wt f:\wt U\to \wt{h(U)}$ of $f$.  Let $\wh U$ (respectively $\wh{h(U}$) be the closure of $\wt U$ (respectively $\wt{h(U}$)  in $\D^{2}$  and let $\wh f:\wh U\to\wh{h(U)}$ be the extension of $\wt f$ to $\wh U$.  Let 
$$X = |\Lambda^{\sigma}_{+}|\cap|\Lambda^{\sigma}_{-}|, \qquad X^{*} = |\Lambda^{f_{*}(\sigma)}_{+}|\cap|\Lambda^{f_{*}(\sigma)}_{-}|$$  and $\wt X$ (respectively $\wt X^{*}$) the  lift of   $X$ to $\wt U$ (respectively $X^{*}$ to $\wt{h(U)}$). Define $\wt h:\wt X\to\wt X^{*}$ as follows. If $\wt x = \wt \lambda_{-}\cap\wt\lambda_{+}$ with $\wt\lambda_{-}\in\wt\Lambda_{-}$ and $\wt\lambda_{+}\in\wt\Lambda_{+}$, define $\wt h(\wt x) = \wt f_{*}(\wt\lambda_{-})\cap\wt f_{*}(\wt\lambda_{+})$.

The fact that $\wh f:\wh U\to\wh{h(U)}$ is continuous implies that the map $\wt h:\wt X\to\wt X^{*}$ is a homeomorphism.    As in   Casson-Bleiler~\cite[pp.~89-90]{bca}, we extend $\wt h$ linearly and equivariantly over any lift of an arc of $|\wt\Lambda^{\sigma}_{+}|\cup|\wt\Lambda^{\sigma}_{-}|$ with both endpoints in $\wt X$ and interior disjoint from $\wt X$ and equivariantly over the  lifts of rectangular components of $U\sm(|\Lambda^{\sigma}_{+}|\cup|\Lambda^{\sigma}_{-}|)$   using the technique of  Casson-Bleiler~\cite[pp.~90]{bca}. It follows from Corollary~\ref{compcompl} that the closures of the  nonrectangular components of $U\sm(|\Lambda^{\sigma}_{+}|\cup|\Lambda^{\sigma}_{-}|)$ consist of finitely many sets which are either,
\begin{enumerate}

\item Annuli that are the nuclei of pairs of complementary crowns with one edge a rim of the crown\upn{;}\label{first}

\item Disks that are the nuclei of pairs of complementary crown sets\upn{;}

\item  Punctured disks that are the nuclei of pairs of complementary crown sets\upn{;}

\item M\"obius strips that are the nuclei of pairs of complementary crown sets\upn{;}\label{fourth}

\item Annuli formed by pasting   the nuclei of two pairs of complementary crowns   together along their common rims.\label{fifth}

\end{enumerate}

The rims of the crown sets in cases~(\ref{first}) and~(\ref{fifth}) and the center geodesic of the M\"obius strips in case~(\ref{fourth}) belong to $\Gamma$  so the homeomorphism $h$ has already been defined above on these geodesics. In all cases, the homeomorphism $\wt h$ can be extended equivariantly in an  arbitrary way over the lifts of the nonrectangular components of $U\sm(|\Lambda^{\sigma}_{+}|\cup|\Lambda^{\sigma}_{-}|)$.

Both $\wh f,\wh h:\wh U\to \wh{h(U)}$ are defined and agree on $\Se$. Thus,~\cite[Corollary~5]{cc:epstein} implies that the homeomorphisms $h$ and $f$ are isotopic.  
\end{proof}

\subsection{Completing the proof of Theorem~\ref{NTthm}}

The next two lemmas complete the proof of Theorem~\ref{NTthm}. In the lemmas, $U$ is a component of $L\sm|\Gamma|$, $S$ is the internal completion of $U$, and $n_{S} = n_{U}$. The map $h:U\to h(U)$ extends in a natural way to a map $h:S\to h(S)$ where $h(S)$ is the internal completion of $h(U)$.

\begin{lemma}

If $S$ is a pseudo-Anosov component, the homeomorphism $h^{n_{S}}:S\to S$ satisfies \emph{Theorem~\ref{NTthm}~(\ref{NT2})}.

\end{lemma}

\begin{proof}

Since $h:U\to h(U)$ satisfies $h(\Lambda^{\sigma}_{\pm}) = \Lambda^{f_{*}(\sigma)}_{\pm}$, it follows that $h:S\to h(S)$ satisfies $h(\Lambda^{\sigma}_{\pm}) = \Lambda^{f_{*}(\sigma)}_{\pm}$ and $h^{n_{S}}:S\to h^{n_{S}}(S) = S$ satisfies $h^{n_{S}}(\Lambda^{\sigma}_{\pm}) = \Lambda^{f^{n_{S}}_{*}(\sigma)}_{\pm}$.  But by Proposition~\ref{lamuniq}, $\Lambda^{f^{n_{S}}_{*}(\sigma)}_{\pm} = \Lambda^{\sigma}_{\pm}$.
\end{proof}

\begin{lemma}

If $S$ is a periodic component, the homeomorphism $h^{n_{S}}:S\to S$ satisfies \emph{Theorem~\ref{NTthm}~(\ref{NT1})}.

\end{lemma}

\begin{proof}[Sketch of proof]

First assume that $S$ is not planar with $b+c=3$.  It is elementary to construct two finite sets $\Sigma$ and $\Sigma'$ of pairwise disjoint, simple closed geodesics in $S$ such that the components of $S\sm(|\Sigma|\cup|\Sigma'|)$ consist of finitely many peripheral annuli, finitely many annular neighborhoods of cusps, and finitely many open disks.

Since $S$ is a periodic component, for each $\sigma\in\Sigma\cup\Sigma'$, there exists an integer $k>0$ such that $\sigma_{kn_{S}} = \sigma$. Thus, there exists an integer $n>0$ such that $\sigma_{nn_{S}} = \sigma$, all $\sigma\in\Sigma\cup\Sigma'$.  Denote $(h^{n_{S}}|S)^{n}$ by $g$.  Explicity construct an ambient  isotopy $\Phi:L\times[0,5]\to L$ such that $\Phi^{0} = \id$ and $\Phi^{5}\circ g = \id$ as follows,

\begin{enumerate}

\item Use~\cite[Theorems~2.1 and~3.3]{Epstein:isotopy}  and induction to  define $\Phi^{t}$, $t\in[0,1]$, moving $g(\sigma)$ to $\sigma$ for every   $\sigma\in\Sigma$\upn{;}   

\item Use the techniques of~\cite[Lemmas~4.58 and~4.60]{cc:hm}, to define $\Phi^{t}$, $t\in[1,2]$,   moving $\Phi^{1}\circ g(\sigma)$ to $\sigma$ for every   $\sigma\in\Sigma'$   while keeping each    $\sigma\in\Sigma$ fixed\upn{;}

\item It is elementary to define $\Phi^{t}$,  $t\in[2,3]$, supported in a neighborhood of $|\Sigma|\cup|\Sigma'|$ and fixing $|\Sigma|\cup|\Sigma'|$, such that   $\Phi^{3}\circ g = \id$ on $|\Sigma|\cup|\Sigma'|$\upn{;}

\item Use Alexander's trick to define $\Phi^{t}$, $t\in[3,4]$, so that $\Phi^{4}\circ g = \id$ on the simple connected components of $S\sm(|\Sigma|\cup|\Sigma'|)$\upn{;}

\item It is elementary to define $\Phi^{t}$,  $t\in[4,5]$, supported on the annular components of $S\sm(|\Sigma|\cup|\Sigma'|)$, so that $\Phi^{5}\circ g = \id$ on the annilar components of $S\sm(|\Sigma|\cup|\Sigma'|)$.

\end{enumerate}
Thus, $\Phi^{0}\circ (h^{n_{S}}|S)^{n} = (h^{n_{S}}|S)^{n} $ and $\Phi^{5}\circ (h^{n_{S}}|S)^{n} = \id$.


If $S$ is planar with $b+c=3$, the lemma is routine.
\end{proof}

\providecommand{\bysame}{\leavevmode\hbox to3em{\hrulefill}\thinspace}
\providecommand{\MR}{\relax\ifhmode\unskip\space\fi MR }
\providecommand{\MRhref}[2]{%
  \href{http://www.ams.org/mathscinet-getitem?mr=#1}{#2}
}
\providecommand{\href}[2]{#2}

\end{document}